\documentclass[12pt]{amsart}
\usepackage{amsfonts}
\usepackage{amssymb}
\usepackage{epsfig,graphics}

\usepackage{amscd}

\newcommand{\BZ}{{\mathbb{Z}}}

\newcommand{\BR}{{\mathbb{R}}}
\newcommand{\BC}{{\mathbb{C}}}

\newcommand{\BQ}{{\mathbb{Q}}}

\newcommand{\BG}{{\mathbb{G}}}

\newcommand{\BH}{{\mathbb{H}}}
\newcommand{\gD}{\Delta}

\newcommand{\gC}{\Gamma}
\newcommand{\gc}{\gamma}

\newcommand{\gS}{\Sigma}
\newcommand{\gO}{\Omega}
\newcommand{\gep}{\epsilon}

\newcommand{\gL}{\Lambda}

\newcommand{\ti}[1]{\tilde{#1}}

\newcommand{\SO}{\text{SO}}

\newcommand{\cal}{\mathcal}
\newtheorem{thm}{Theorem}[section]
\newtheorem{prop}[thm]{Proposition}
\newtheorem{lem}[thm]{Lemma}
\newtheorem{cor}[thm]{Corollary}

\newtheorem{defn}[thm]{Definition}
\newtheorem{rem}[thm]{Remark}

\newtheorem{clm}[thm]{Claim}

\begin{document}

\author{N. Bergeron and T. Gelander}

\date{\today}

\title{A note on local rigidity}

\maketitle

\begin{abstract}
The aim of this note is to give a geometric proof for classical local 
rigidity of lattices in semisimple Lie groups. We are reproving well known
results in a more geometric (and hopefully clearer) way.
\end{abstract}


\section{Introduction}

Let $G$ be a semisimple Lie group, and let $\gC\leq G$ be an irreducible lattice.
We denote by $\mathcal{R}(\gC ,G)$ the space of deformations of $\gC$ in $G$,
i.e. the space of all homomorphisms $\gC\to G$ with the topology of pointwise
convergence. 
Recall that $\gC$ is called {\it locally rigid} if there is a neighborhood $\gO$
of the inclusion map $\rho_0:\gC\to G$ in 
$\mathcal{R}(\gC ,G)$ such that any $\rho\in\gO$ is conjugate to $\rho_0$.
In other words, this means that if $\gC$ is generated by the finite set $\gS$,
then there is an identity neighborhood $U\subset G$ such that if 
$\rho :\gC\to G$ is a homomorphism such that $\rho (\gc )\in \gc\cdot U$ for
any $\gc\in\gS$, then there is some $g\in G$ for which 
$\rho (\gc )=g\gc g^{-1}$ for any $\gc\in\gC$ (such $\rho$ is called a trivial deformation).

Local rigidity was first proved by Selberg \cite{Selberg} for uniform
lattices in the case $G=SL_n(\BR ),~ n\geq 3$, and by Calabi \cite{Calabi}
for uniform lattices in the case $G=PO(n,1)=\mbox{Isom}(\BH^n)$, $n\geq 3$. 
Then, Weil \cite{Weil}
generalized these results to any uniform irreducible lattice in any $G$, assuming 
that $G$ is not locally isomorphic to $SL_2(\BR)$ (in which case lattices 
have many non-trivial deformations - the Teichmuller spaces). Later, Garland and 
Raghunathan \cite{Gar-Rag} proved local rigidity for non-uniform lattices 
$\gC\leq G$ when $\mbox{rank} (G)=1$ \footnote{Unless otherwise specified $\mbox{rank} (G)$ will 
always mean the real rank of $G$.} under the necessary assumption that $G$ 
is neither locally isomorphic to $SL_2(\BR )$ nor to $SL_2(\BC )$. For non-uniform
irreducible lattices in higher rank semisimple Lie groups, the local rigidity is a 
consequence of the much stronger property: super-rigidity, which was proved by Margulis 
\cite{Margulis}.

Our aim here, is to present a simple proof of the following theorem:

\begin{thm}\label{thm1}
Under the necessary conditions on $G$ and $\gC$, if $\rho$ is a deformation,
close enough to the identity $\rho_0$, then $\rho (\gC)$ is a lattice and 
$\rho$ is an isomorphism.
\end{thm}

\begin{rem} 
In the higher rank case, we need the lattice to be arithmetic, which 
is always true by Margulis' theorem. But this is of course cheating. Anyway our 
method provides, for example, a simple geometric proof of the local rigidity of 
$SL_n ({\Bbb Z} )$ in $SL_n ({\Bbb R})$ for $n\geq 3$. Note also that Margulis' arithmeticity 
theorem is easier in the non compact case \cite {M3}
(it does not rely on super-rigidity) which is the only case 
we need here.   
\end{rem}

\begin{rem}
In fact, we prove a ``stronger'' result than the statement of theorem 
\ref{thm1}. Namely, if $X=G/K$ is the associated symmetric space, then 
for a small deformation $\rho$, $X/\rho (\gC )$ is homeomorphic to $X/\gC$ and 
$\rho :\gC\to\rho (\gC )$ is an isomorphism.
\end{rem} 

With Theorem \ref{thm1}, local rigidity follows from Mostow rigidity.
Recall that Mostow's rigidity theorem \cite{Mostow} says that if $G$ is a connected
semisimple Lie group not locally isomorphic to $SL_2(\BR )$, and $\gC_1,\gC_2\leq G$
are isomorphic irreducible lattices in $G$, then they are conjugate in $G$, i.e. 
$\gC_1=g\gC_2g^{-1}$ for some $g\in G$.

\begin{rem}
In Selberg's paper \cite{Selberg} a main difficulty 
was to prove a weaker version of \ref{thm1} in his special case.
Selberg also indicated that his proof (which is quite elementary and very 
elegant) would work for more general higher rank
cases if the general version (\ref{thm1} above) of his lemma 9 could be proved.
This supports the general feeling that it should be possible to give a complete elementary proof 
of local rigidity which does not rely on Mostow's rigidity.
However, the authors decided not to strain too much, in trying to avoid Mostow 
rigidity, since in the rank one case there are by now very elegant and straightforward 
proofs of Mostow rigidity (such as the proof of Gromov and Thurston 
(see \cite{Thurston}) for compact hyperbolic manifolds, and of Besson, Courtois and 
Gallot \cite{BCG} for the general rank one case, note also that Mostow's original proof is much simpler
in the rank one case) while for higher rank spaces
the much stronger Margulis' supper-rigidity holds.
\end{rem}

Our proof of Theorem \ref{thm1} relies on an old principle implicit in Ehresmann's work 
and to our knowledge first made explicit in Thurston's notes. This principle implies 
in particular that any sufficiently small deformation of a cocompact lattice in a 
Lie group stays a lattice. It seems to have remained unknown for many years, as both Selberg 
and Weil spent some effort in proving partial cases of it. It is by now fairly well known, but 
the authors could not find a complete written proof of it in the literature although 
everyone knows that it is implicit in Ehresmann's work. In the second section of this paper, we 
thus decided to write a complete proof, using Ehresmann's beautiful viewpoint, 
of a  slight generalization of this principle to the non-compact case needed for our purpose.  
In the third section we prove Theorem \ref{thm1} in the rank one case.
In the last section we prove Theorem \ref{thm1} for arithmetic lattices of 
higher rank.

Non-uniform lattices in the group $PSL_2(\BC )\cong\mbox{Isom}(\BH^3 )$ are 
not locally rigid, although Mostow rigidity is valid for $PSL_2(\BC )$. This
shows that local rigidity is not a straightforward consequence of Mostow 
rigidity. Moreover, it was shown by Thurston that if $\gC\leq PSL_2(\BC )$ is
a non-uniform lattice then there is an infinite family of small deformations
$\rho_n$ of $\gC$ such that $\rho_n(\gC )$ is again a lattice, but not 
isomorphic to $\gC$, and in fact many of the $\rho_n(\gC )$'s may be constructed to be uniform. 
The proof which we present in this note is in a sense a product of our 
attempt to understand and formulate why similar phenomenon cannot happen in 
higher dimensions.

In fact, we shall show that some weak version of local rigidity is valid also 
for non-uniform lattices in $PSL_2(\BC )$. This would allow us to obtain a direct proof of the fact
that for any (non-uniform) lattice in $G=SL_2(\BC )$ (as well as in any connected $G$ which is defined 
over $\BQ$ as an algebraic group and is locally isomorphic to $SL_2(\BC )$), there is an algebraic 
number field $K$, and an
element $g\in G$, such that $g\gC g^{-1}\leq G(K)$. This fact (which is well known in the general
case for any lattice in any $G$ which is not locally isomorphic to $SL_2(\BR )$) can be proved by 
a few line argument when $\gC$ is locally rigid in $G$. When $G$ is locally isomorphic to 
$SL_2(\BC )$, and $\gC\leq G$ is non-uniform (and hence not locally rigid), this was proved in 
\cite{Gar-Rag}, section 8, by a special and longer argument. Our method allows to unify this 
special case to the general case. 

\medskip

In section \ref{Ehresmann-Thurston} we shall formulate and prove the Ehresmann-Thurston principle and 
use it in order to prove local rigidity for uniform lattices (Weil's theorem). In section \ref{rank-1},
which is in a sense the main contribution of this note, we shall present a complete and elementary 
proof for local rigidity of non-uniform lattices in groups of 
rank one (originally due to Garland and Raghunathan).
Then, in section \ref{higher-rank}, we shall apply our method to higher-rank non-uniform lattices.
In some cases, our proof remains completely elementary. However, for the general case we shall rely on some
deep results.


\section{The Ehresmann-Thurston principle}\label{Ehresmann-Thurston}

In this section $X$ is a model manifold and $G$ a transitive group of real analytic homeomorphisms 
of $X$. Given such data one can construct all possible manifolds by choosing open subsets of $X$ 
and pasting them together using restrictions of homeomorphisms from the group $G$. 
Following Ehresmann and Thurston
we will call such a manifold a $(G,X)$-manifold. More precisely, a {\it $(G,X)$-manifold} is a 
manifold $M$ together with an atlas $\{\kappa_i : U_i \rightarrow X\}$ 
such that the changes of charts are restrictions of elements of $G$. 

Such a structure enables to develop the manifold $M$ along its paths 
(by pasting the open subsets of X) into the model manifold $X$. We can do this in 
particular for closed paths representing generators of the fundamental group $\pi_1(M)$. 
Starting from one of the open sets $U\subset M$ of the given atlas, 
the development produces a covering space $M'$ of $M$, a 
representation $\pi_1(M) \stackrel{\rho}{\rightarrow} G$ called the holonomy and a 
structure preserving immersion, the developing map $M' \stackrel{d}{\rightarrow} X$, which is 
equivariant with respect to the action of $\pi_1(M)$ on $M'$ and of 
$\rho\big(\pi_1(M)\big)$ on $X$. 

There is an ambiguity in the choice of one chart $\kappa : U\rightarrow X$ about a base point.
However, deferent charts around a point are always identified on a common sub-neighborhood modulo
the action of $G$.

Ehresmann's way of doing mathematics was to find neat definitions for his concepts so that 
theorems would then follow easily
from definitions. We shall try to follow his way. Recall first his fiber bundle picture of 
a $(G,X)$-structure and his neat definition of the developing map.

To make it simple, let us describe first the picture for one chart $\kappa : U \rightarrow X$ which 
is a topological embedding. We can associate a trivial fiber 
bundle $E_U = U \times X \rightarrow U$ by assigning to any $m \in U$ the model space $X_m =X$. 
The manifold $U$ is embedded as the
diagonal cross section $s(U) =\{ (m,\kappa (m)) \; : \; m\in U \} \subset U \times X$. Its points 
are the points of tangency of fibers and base manifolds.

For the global picture of a $(G,X)$-structure on a manifold $M$, we will restrict ourselves
to compatible charts that are topological 
embeddings $\kappa : U \rightarrow X$ for open sets $U \subset M$. A point of the fiber bundle 
space $E$ over $M$ is, by definition, a triple 
$$ 
 \{ m, \kappa , x\},
$$
where $m\in U \subset M$, $\kappa : U\rightarrow X$ is a compatible chart and $x \in X$, modulo 
the equivalence relation given by the action of 
$G$, i.e. $\{ m, \kappa , x\} \sim \{m, \kappa ' ,x' \}$ 
iff there is $g\in G$ such that $x'=g\cdot x$ and $\kappa '=g\circ\kappa$ 
on some sub-neighborhood $V\subset U\cap U'$.
In $E$, the manifold $M$ is embedded as the diagonal cross section $s(M)$, whose points are 
represented by triples $\{ m, \kappa , \kappa (m) \}$. The horizontal leaves represented by the 
triples $\{U, \kappa , x \}$ give a foliation 
${\cal F}$ of the total space $E$, which induces an $n$-plane field $\xi$ in $E$ 
transversal to the fibers and transversal to the cross section $s(M)$. 
Given such data, Ehresmann has defined the classical notion of holonomy. The holonomy is 
obtained by lifting a closed curve starting 
and ending at $m_0 \in M$, into all curves tangent to $\xi$, and by this getting an element of
$G$ acting on $X_{m_0}$. 
For contractible closed curves in the base space $M$, the holonomy is of course the identity map 
of the fiber $X_{m_0}$.

This leads us, following Ehresmann, to consider the general notion of {\it flat Cartan connections}. 
In fact $(G,X)$-manifolds are the flat cases of manifolds $M$ with a general $(G,X)$-connection. 
A {\it $(G,X)$-connection} is defined in \cite{Ehresmann2} as follows:
\begin{enumerate}
\item a fiber bundle $X \rightarrow E \rightarrow M$ with fiber $X$ over $M$,
\item a fixed cross section $s(M)$,
\item an $n$-plane field $\xi$ in $E$ transversal to the fibers and transversal to the fixed cross 
section, such that
\item the holonomy of each closed curve starting and ending at $m_0 \in M$ is obtained by lifting 
it to all curves tangent to $\xi$, and is a homeomorphism of $X_{m_0}$ which is induced by an 
element of $G$.
\end{enumerate} 
A $(G,X)$-connection is said to be {\it flat} if contractible closed curves have trivial holonomy.

We have just seen that a $(G,X)$-structure on a manifold $M$ yields a flat $(G,X)$-connection on 
$M$. Let us describe how we can recover the holonomy and the developing map of a given 
$(G,X)$-structure from the induced flat $(G,X)$-connection.
Let $M$ be a manifold with a flat $(G,X)$-connection. For general closed curves, starting and 
ending at $m_0$, the holonomy gives a representation of $\pi_1 (M)$ into the group $G$ acting 
on $X_{m_0}$. When the $(G,X)$-connection is induced by a $(G,X)$-structure on $M$, this map is 
the usual holonomy map of the $(G,X)$-structure. 
$E$ can be described as the quotient
$$
 \pi_1(M) \backslash (M' \times X) 
$$
where the action is the diagonal one, with $\pi_1(M)$ acting on $M'$ by covering 
transformations and on $X$ via the holonomy map through the action of $G$ on $X$.
The development of a curve ending at $m_0\in M$, is obtained by dragging along $\xi$ the 
corresponding points until they arrive in the fiber $X_{m_0}$. 
As soon as the connection $\xi$ is flat, homotopic curves with common initial and end points give 
the same image of the initial point in the end fiber and the development map 
$M' \rightarrow X_{m_0}$ is well defined. We thus conclude that 
any flat $(G,X)$-connection on a manifold $M$ yields a $(G,X)$-structure.

It is then an easy matter to prove the following useful result, which to our knowledge first 
appeared in a slightly different setup in Thurston's notes \cite{Thurston}, but which, as we will see, and   
seems to be a common knowledge, is elementary using Ehresmann's viewpoint of $(G,X)$-structures. 
For a slightly different proof, from which we have borrowed some ideas, see 
\cite{CanaryEpsteinGreen} which is a very nice reference for all basic material on 
$(G,X)$-structures.

\begin{thm}[{\bf The Ehresmann-Thurston Principle}]
\label{prin}
Let $N$ be a smooth compact manifold possibly with boundary.
Let $N_{Th}$ be the reunion of $N$ with a small collar $\partial N  \times [0,1)$ of its boundary. 
Assume $N_{Th}$ is equipped with a $(G,X)$-structure $M$ whose holonomy is 
$\rho_0 : \pi_1(N)\rightarrow G$.
Then, for any sufficiently small deformation $\rho$ of $\rho_0$ in $\mathcal{R} ( \pi_1(N),G)$, there is a  
$(G,X)$-structure on the interior of $N$ whose holonomy is $\rho$. 
\end{thm}

\begin{proof}
Let $\rho \in \mathcal{R} (\pi_1(N) , G)$ be a deformation of $\rho_0$. 
We define the following two fiber bundles over $N_{Th}$ with fibers $X$
$$E_{\rho_0} = \pi_1(N) \backslash (N_{Th}' \times X)$$
and 
$$E_{\rho}  = \pi_1(N) \backslash (N_{Th}' \times X)$$
where the action of $\pi_1(N)$ on $X$ is respectively induced via $\rho_0$ and $\rho$ by the 
natural action of $G$ on $X$.   
We denote by $(1,\rho_0 )$ and $(1, \rho)$ the two diagonal actions of $\pi_1(N)$ on 
$N_{Th} \times X$ considered above.

Note that since $N$ is compact, $\pi_1(N)$ is finitely generated and hence \
$\mathcal{R}\big(\pi_1(N),G\big)$ has the structure of an analytic manifold and, in particular, it
is locally arcwise connected. Theorem \ref{prin} is a consequence of the following easy claim.

\begin{clm}\label{diffeo}
If $\rho$ is a deformation of $\rho_0$ in the same connected component of 
$\mathcal{R}\big(\pi_1(N),G\big)$ then there exists a fiber bundle map 
$F: E_{\rho_0} \rightarrow E_{\rho}$ such that
\begin{enumerate}
\item $F$ restricts to a diffeomorphism $\Phi$ above $N$,
\item as $\rho$ gets closer to $\rho_0$,  the lifted diffeomorphism $\tilde{\Phi} : N ' \times X \rightarrow N' \times X$ 
gets closer to the identity in the compact-open topology.
\end{enumerate}
\end{clm}

We temporarily postpone the proof of the claim and conclude the proof of Theorem \ref{prin}. 

When $\rho$ is in the same connected component as $\rho_0$, we have, by Claim \ref{diffeo}, a 
diffeomorphism $\Phi$ 
between the compact fiber bundles $E_{\rho_0}^c$ and $E_{\rho}^c$ which are defined over $N$.
Thus, if ${\cal F}$ denotes 
the horizontal foliation in $E_{\rho_0}$, the restriction of 
${\cal F}$ to $E_{\rho_0}^c$ induces via $\Phi$ an $n$-plane field $\xi$ on 
$E_{\rho}^c$. When $\rho$ is close enough to $\rho_0$, the lifted diffeomorphism $\tilde{\Phi}$ is 
close to the identity so that the $n$-plane field $\xi$ is transversal to both the fibers and the image by $\Phi$ of 
the diagonal cross section in $E_{\rho_0}^c$. Moreover, the holonomy which is
obtained by lifting a closed 
curve starting and ending at $m_0 \in N$, to all curves tangent to $\xi$, is equal
to the image by $\rho$ in $G$ of the homotopy class of the curve. Thus the restriction of 
the bundle $E_{\rho}^c$ to a bundle over the interior of $N$ yields a $(G,X)$-connection. 
This connection is obviously flat and hence gives a 
$(G,X)$-structure on the interior of $N$ whose holonomy is $\rho$. This concludes the proof of 
Theorem \ref{prin} modulo Claim \ref{diffeo}.

\medskip

Let us now prove Claim \ref{diffeo}.
Let $\{ U_i \}_{0 \leq i \leq k}$ be a finite open covering of $N_{Th}$ such that
\begin{itemize}
\item $U_0 = N_{Th} -N$, 
\item for each $1\leq i \leq k$, $U_i$ is simply connected and 
the fiber bundles $E_{\rho}$ and $E_{\rho_0}$ are both trivial above $U_i$.
\end{itemize}

For each integer $1 \leq i \leq k$, we fix a trivialization $(E_{\rho})_{|U_i} \cong U_i \times X$ (resp. 
$(E_{\rho_0})_{|U_i} \cong U_i \times X$) of $E_{\rho}$ (resp.  
$E_{\rho_0}$) above $U_i$. For each pair of integers $1 \leq i\neq j \leq k$, we then denote by $g_{\rho}^{i,j}: 
(U_i \cap U_j ) \times X \rightarrow (U_i \cap U_j ) \times X$ 
(resp. $g_{\rho_0}^{i,j}:(U_i \cap U_j ) \times X \rightarrow (U_i \cap U_j )\times X $)
the diffeomorphism, which is the product of the identity on the first factor with the corresponding
change of charts on the second factor,
between the trivializations $(E_{\rho})_{|U_i}$ and $(E_{\rho})_{|U_j}$ (resp. 
$(E_{\rho_0})_{|U_i}$ and $(E_{\rho_0})_{|U_j}$).

Let $U_i^1 = U_i$ $(1 \leq i \leq k)$ and for each integer $r>0$, let 
$\{ U_i^{r+1} \}_{1\leq i \leq k}$ be a shrinking of 
$\{ U_i^r \}_{1 \leq i \leq k}$ (i.e. for each integer $i$, $U_i^{r+1}$ is an open set whose 
closure is included in $U_i^r$) 
such that $\{ U_0 \} \cup \{ U_i^{r+1} \}_{1\leq i \leq k}$ is still a covering of $N_{Th}$.

As both $E_{\rho}$ and $E_{\rho_0}$ are trivial above $U_1$, the identity map $U_1 \times X \rightarrow U_1 \times X$
induces a diffeomorphism $F_1:(E_{\rho_0})_{U_1^1}\to (E_{\rho})_{U_1^1}$. 

We will define by induction on $s$, a diffeomorphism $F_s$ between $E_{\rho_0}$ and $E_{\rho}$ 
above $U_1^s \cup \ldots \cup U_s^s$
which is equal to $F_{s-1}$ above $U_1^s \cup \ldots \cup U_{s-1}^s$. But for the sake of clarity 
let us first define the diffeomorphism $F_2$.

Let's first describe the $s=2$ step.   
First note that, as $\rho$ is a deformation of $\rho_0$ in the same connected component, the diffeomorphism 
$g_{\rho}^{1,2} \circ (g_{\rho_0 }^{1,2} )^{-1}$ of $(U_1 \cap U_2 ) \times X$ 
is isotopic to the identity, by an isotopy which preserves each of the fiber $\{ *\} \times X$. We denote by
$(m,x) \mapsto (m, \varphi_t (x))$, $t\in [0,1]$, this isotopy.

We want to define $F_2$ above $U_2^2$, recall it will be equal to $F_1$ above $U_2^2\cap U_1^2$. 
Let $W$ be an open set such that
$$\overline{U_2^2} \subset W \subset \overline{W} \subset U_2^1 .$$
We define $F_2$ above $U_2^2$ as follows:
\begin{itemize}
\item 
Above $U_2^1 - \overline{W} \subset U_2$, both $E_{\rho_0}$ and $E_{\rho}$ are trivial and the identity map 
$(U_2^1 - \overline{W} ) \times X \rightarrow (U_2^1 -\overline{W}) \times X$ induces a diffeomorphism
$f : (E_{\rho_0})_{|(U_2^1 -\overline{W})} \rightarrow (E_{\rho})_{|(U_2^1 -\overline{W})}$. 
\item 
We define $f=F_1$ above $U_1^2 \cap U_2^2$. 
(In the coordinates $(U_1^2 \cap U_2^2 ) \times X$ induced
by the trivializations above $U_2$, the diffeomorphism $f$ is equal to the restriction of 
$g_{\rho}^{1,2} \circ (g_{\rho_0 }^{1,2} )^{-1}$.)
\item 
We extend $f$ above a small collar neighborhood $B \times [0, \varepsilon )$ of the boundary
$B$ of $U_1^2 \cap U_2^2$ in $W-(U_1^2 \cap U_2^2)$ by the map from 
$\big( B \times [0, \varepsilon )\big) \times X$ to itself, given by
$\big( (b,t),x\big)\mapsto\big( (b,t) , \varphi_{g(t)}(x)\big)$
where $g:[0,\gep ]\to [0,1]$ is the classical smooth bump function.
\item 
Finally we can extend $f$ (by the identity) above all $U_2^1$ to get a smooth diffeomorphism.
\item 
We then define $F_2$ above $U_2^2$ to be the restriction of $f$ to that set.
\end{itemize}
If we let $F_2$ to be equal to $F_1$ above $U_1^2$, we get a well defined diffeomorphism $F_2$ from 
$E_{\rho_0}$ to $E_{\rho}$ above $U_1^2 \cup U_2^2$.

To follow the induction we need to define $F_{s+1}$ above $U_{s+1}^{s+1}$. First note that above $U_{s}^{s+1}$, the fiber bundles 
$E_{\rho_0}$ and $E_{\rho}$ are both trivial. 

Let $W$ be an open set such that
$$\overline{U_{s+1}^{s+1}} \subset W \subset \overline{W} \subset U_{s+1}^s .$$
We define $F_{s+1}$ above $U_{s+1}^{s+1}$ as follows:
\begin{itemize}
\item Above $U_{s+1}^s - \overline{W}$, both $E_{\rho_0}$ and $E_{\rho}$ are trivial and we define a function $f$ to be the canonical 
diffeomorphism between their trivializations.
\item Above $U_{s+1}^{s+1} \cap (U_1^{s+1} \cup \ldots \cup U_s^{s+1} )$ we define $f$ to be equal to $F_s$.
\item We extend $f$ as in the first inductive step, by using an isotopy between changes of charts 
composed with a smooth bump function, 
to get a smooth diffeomorphism between $E_{\rho_0}$ and $E_{\rho}$ above all $U_{s+1}^s$.
\item We define $F_{s+1}$ above $U_{s+1}^{s+1}$ to be equal to the restriction of $f$ above $U_{s+1}^{s+1}$.
\end{itemize}
If we let $F_{s+1}$ to be equal to $F_s$ above $U_1^{s+1} \cup \ldots \cup U_s^{s+1}$,
we get a well defined diffeomorphism $F_{s+1}$ from $E_{\rho_0}$ to $E_{\rho}$ above $U_1^{s+1} \cup \ldots \cup U_{s+1}^{s+1}$.

The local trivialization of the fiber bundle $E_{\rho}$ depends continuously (in the $C^{\infty}$-topology) on $\rho$, and hence
$F_{s+1}$ depends continuously on $F_s$ and $\rho$. 

We continue the induction until $s=k$ and get a diffeomorphism between $E_{\rho_0}$ and $E_{\rho}$ 
above $U_1^{k} \cup \ldots \cup U_k^k$ which 
restricts to a fiber bundle diffeomorphism between the corresponding bundles above $N$. Moreover we have seen that the 
restriction of $F$ above $N$ depends continuously (in the $C^{\infty}$ topology) on $\rho$. 
This concludes the proof of the claim. 
\end{proof}

\bigskip

Recall now  Ehresmann's definition \cite{Ehresmann2} of a complete $(G,X)$-structure on a manifold $M$. Let $M$ be a $(G,X)$-manifold.
Any curve $c$ starting from a point $m_0$ in $M$ can be developed to a curve $\tilde{c}$ in $X$. 
The $(G,X)$-manifold is said to be {\it complete} if, conversely, any curve $\tilde{C}$ extending 
$\tilde{c}$ in $X$ is a development of a curve in $M$ extending $c$.
When $X$ is a Riemannian homogeneous $G$-space, for a Lie group $G$ of isometries of $X$, 
it is a classical theorem of Hopf and Rinow that this notion of completeness coincides with the 
usual notion of metric completeness.
  
The Ehresmann-Thurston principle above is especially useful when combined with the following.

\begin{thm}[Ehresmann \cite{Ehresmann1}] \label{complet}
If $M$ is a complete $(G,X)$-manifold, then the developing map induces an isomorphism between its 
universal cover and the universal cover of $X$.
\end{thm} 

As a corollary of Theorem \ref{prin} and Theorem \ref{complet} we get the following classical result of Weil 
\cite{Weil0} (see also \cite{Weil}).

\begin{cor} \label{compact rigidity}
Let $\Gamma$ be a cocompact lattice in a connected center free semi-simple Lie group without
compact factors $G$. 
If $\rho$ is a deformation of $\Gamma$ in $G$, close enough to the identity, then $\rho ( \Gamma )$ is still a cocompact lattice 
and is isomorphic to $\Gamma$.
\end{cor}

\begin{proof}[Proof of Corollary \ref{compact rigidity}.] 
The group $\gC$ is finitely generated and by Selberg's lemma it has a torsion-free finite index subgroup. 
It is easy to see that if \ref{compact rigidity} holds for a finite index subgroup of $\gC$ then it
also holds for $\gC$. Hence, we shall assume that $\gC$ itself is torsion-free. 
Let $X$ be the symmetric space 
associated to $G$. It is simply connected. Let $M$ be the $(G,X)$-manifold $\Gamma \backslash X$. 
It is compact and its fundamental group is $\Gamma$. According to Theorem \ref{prin}, if 
$\rho$ is a sufficiently small deformation of the inclusion $\Gamma \subset G$, then there exists a new 
$(G,X)$-structure $M'$ on the manifold $M$, whose holonomy is $\rho : \Gamma \rightarrow G$. 
But as $M$ is compact and $X$ is Riemannian, this new $(G,X)$-structure is complete and, by Theorem
\ref{complet}, the developing map is a $\Gamma$-equivariant isomorphism 
between its universal cover and $X$. This implies that $\rho (\Gamma )$ acts properly 
discontinuously on $X$ and that 
$M'= \rho (\Gamma ) \backslash X$. As $M'$ is homeomorphic to $M$ this concludes the proof.  
\end{proof}

We shall conclude this section by stating the following natural generalization of Corollary 
\ref{compact rigidity} which is also due to Weil \cite{Weil0}.

\begin{prop}\label{GWT}
Let $G$ be a connected Lie group and $\gC\leq G$ a uniform lattice. Then for any sufficiently
small deformation $\rho$ of $\gC$ in $G$, $\rho (\gC )$ is again a uniform lattice and 
$\rho :\gC\to\rho (\gC )$ is an isomorphism.
\end{prop}

\begin{proof}
Assume first that $G$ is simply connected. Then we can chose a left invariant Riemannian metric
$m$ on $G$, and argue verbatim as in the proof Corollary \ref{compact rigidity} letting $(G,m)$
to stand for $X$, and $G$ to act isometrically by left multiplication. 

Now consider the general case. 
Let $\ti G$ be the universal covering of $G$. Let $\gC\leq G$ be a uniform lattice and let $\ti\gC$
be its pre-image in $\ti G$. Then $\ti\gC$ is a uniform lattice in $\ti G$, and $G/\gC$ is 
homeomorphic to $\ti G/\ti\gC$. Since $\gC$ is finitely generated, the deformation space 
$R(\gC ,G)$ has the structure of a manifold and in particular it is locally arcwise connected. 
Therefore if $\rho$ is sufficiently small deformation of $\gC$ in $G$ then there is a curve of deformations 
$\rho_t\in R(\gC ,G)$ with $\rho_0=$ the identity, and $\rho_1=\rho$. Now for any 
$\gc\in\gC$ and any $\ti\gc\in\ti\gC$ above $\gc$, the curve $\rho_t(\gc )$ lifts uniquely to a 
curve $\ti\rho_t(\ti\gc )$. The uniqueness guaranties that $\ti\rho_t$ is again a homomorphism for any $t$.
Hence $\rho_t$ lifts to a curve 
$\ti\rho_t\in R(\ti\gC ,\ti G)$.
In particular $\rho=\rho_1$ lifts to a deformation $\ti\rho=\ti\rho_1$ of $\tilde{\gC}$ in $\ti G$.
It is also easy to see that the lifting $\rho\mapsto\ti\rho$ on a neighborhood of the inclusion $\rho_0$
is continuous as a map from an open set in $R(\gC ,G)$ to $R(\ti\gC,\ti G)$.
Thus, the general case follows from the simply connected one.
\end{proof}

\bigskip

The proof of Theorem \ref{thm1} will follow the same lines, we just need to understand the 
non-compact parts, namely the ends.


\section{The rank one case}\label{rank-1}

Let $G$ be a connected center free simple Lie group of rank one, with associated 
symmetric space $X=G/K$. Let $\gC$ be a non-uniform lattice in $G$. We are 
going to investigate the possible small deformations of $\gC$, and to prove
local rigidity when $n= \dim (X)\geq 4$, and some weak version of it for dimensions $2$ and $3$.

\medskip

By Selberg's lemma $\gC$ is almost torsion free.        
Note that a small deformation which stabilizes a finite index subgroup must 
be trivial. In the sequel we shall assume that $\gC$ is torsion free.

We keep the notation $\rho_0:\gC\to G$ for the inclusion, and $\rho$ for a 
small deformation of $\rho_0$. We denote by $M=\gC\backslash X$ the corresponding locally
symmetric manifold. We shall show that under some conditions on $\rho$,
$\rho (\gC)$ is again a lattice in $G$ and $M'=\rho (\gC)\backslash X$ is homeomorphic
to $M$, and that these conditions are fulfilled when $\dim X\geq 4$.

Our strategy is to use Ehresmann-Thurston's principle. As $M$ is not compact, we shall
decompose it to a compact part $M_0$ which ``exhausts most of $M$'' and to
cusps. More precisely, let $M_0\subset M$ be a fixed compact
submanifold with boundary, which is obtained by cutting all the cusps of $M$ 
along horospheres. Each such cusp is contained in a connected component of the thin part of
the thick-thin decomposition of $M$ corresponding to the constant $\gep_n$ of
the Margulis lemma (see \cite{ThBook} or \cite{BGS}), and is homeomorphic to $T\times\BR^{\geq 0}$ for
some $(n-1)$-dimensional compact manifold $T$. 
We shall call such cusps canonical. More precisely, a {\it canonical cusp}
is a quotient of a horoball by a discrete group of parabolic isometries which
preserves the horoball and acts freely and cocompactly on its boundary 
horosphere. 

Notice that a canonical cusp has finite volume.
To see this, one can look at the Iwasawa decomposition $G=KAN$ which corresponds to a lifting of
the given cusp, express the Haar measure of $G$ in terms of the Haar measures of $K,A$ and $N$,
and estimate the volume of the cusp in the same way as one shows that a Siegel set has a finite
volume. 

In the argument below we shall assume, for the sake of clearness, 
that $M$ has only one cusp, instead of finitely many. 
It is easy to see that the proof works equally well in the general case.
We choose another horosphere which is contained in $M_0$ and parallel to the 
boundary horosphere, and denote the collar between them by $T\times [0,1]$ 
(see figure 1).  

\begin{figure}
\begin{center}
\input{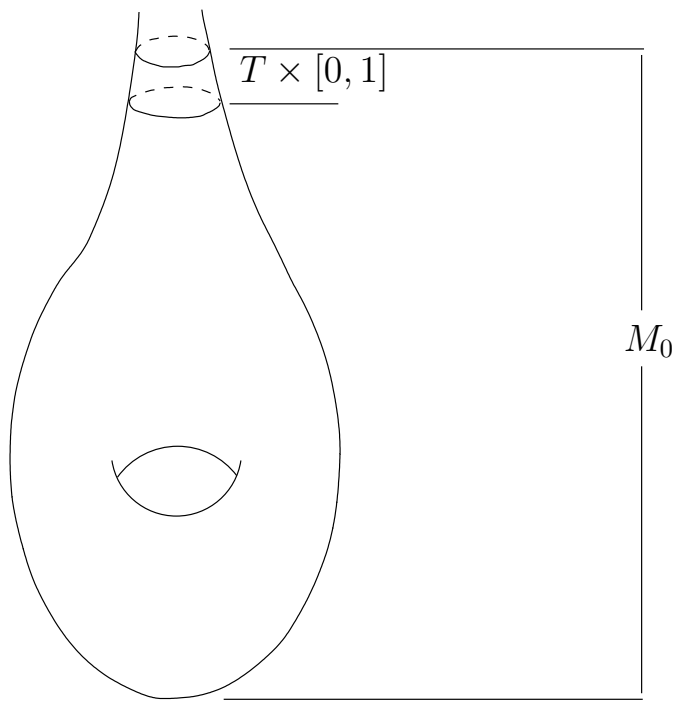}
\caption{The manifold $M_0$}
\end{center}
\end{figure}

By the Ehresmann-Thurston principle, there is a $(G,X)$-structure $M_0'$ on $M_0$ whose 
holonomy is $\rho$.

We will show that under some conditions, the $(G,X)$-structure on $T\times [0,1]$ which is induced 
from $M_0'$ coincides with a $(G,X)$-structure which is induced from some 
canonical cusp $C$ (whose holonomy coincides with $\rho\big(\pi_1(T)\big)$) to
a subset of it, which is homeomorphic to $T\times [0,1]$.
In such case we can glue $M_0$ with $C$ along $T\times [0,1]$ and obtain
a complete $(G,X)$-manifold $M'$ homeomorphic to $M$ whose fundamental group is 
$\rho (\gC )$. Additionally, $M'$ has a finite volume, being a union of the 
compact set $M_0'$ with $C$ which is a canonical cusp and hence has a finite 
volume. 

\medskip

Recall the following basic facts about isometries of rank one symmetric spaces:
\begin{itemize}
\item
Maximal unipotent groups correspond to points in the ideal boundary 
$X(\infty )$. A maximal unipotent group acts simply transitively on each 
horosphere around its fixed point at $\infty$, its dimension is $\dim (X)-1$ and it 
is the unipotent radical of the
parabolic group which is the stabilizer of this point. The maximal
unipotent groups are metabelian.
\item
There are only 3 types of isometries: elliptic, hyperbolic and parabolic.
An elliptic isometry fixes a point in $X$ and is contained in some compact 
subgroup of $G$, it is semisimple and all its eigenvalues have absolute value $1$.

A hyperbolic isometry $\gc$ has exactly two fixed points at the ideal boundary $X(\infty )$ and it
acts by translations on the geodesic connecting them, called the axis of $\gc$, which is also characterized as the
set of points where the displacement function $d(x,\gc\cdot x)$ attains its
minimum. A hyperbolic element is semisimple and has an eigenvalue of absolute value $>1$.

A parabolic element is one which has no fixed points in $X$ and has a unique fixed point in the
ideal boundary $X(\infty )$. If $\gamma$ is parabolic then
$\inf\{d(x,\gc\cdot x):x\in X\}=0$ and the infimum is not attained. An element is parabolic iff it is not semisimple.
In particular, unipotents are paraboloics. In general,
all the eigenvalues of a parabolic element have absolute value $1$.

\item
The quotient map $G\to G/P=X(\infty )$ is open, and the action map
$G\times X(\infty )\to X(\infty )$ is continuous.
\end{itemize}

Let $\gC_T=\pi_1\big (T\times [0,\infty )\big)\leq\gC$ be a subgroup of
$\gC$ which corresponds to the fundamental group of the cusp of $M$. Then
$\gC_T$ is group of parabolic elements, and by \cite{Raghunathan} Corollary 8.25, it contains a subgroup $\gC_T^0$ of finite index
which is unipotent, i.e. a lattice in the unipotent radical of the corresponding parabolic group.
We want to show that $\rho (\gC_T)$ is the
fundamental group of some canonical cusp. So, in particular, we need to
show that $\rho (\gC_T)$ is again a group of parabolic elements. This fact is true only when
$n = \dim X \geq 4$ (but the rest of the argument applies also for dimensions $2$ and $3$).
We shall prove it by
showing that otherwise, $\rho (\gC_T)$ is contained in the stabilizer group of
some geodesic or some point, and then derive a contradiction since
$\rho (\gC_T)$ is ``too large'' and there is ``no room'' for it in such
a subgroup. We shall divide this argument into few claims.

\begin{lem}\label{index}
There exists an integer $i=i(n)$ such that any nilpotent subgroup $A$ of
$\BR\times SO_{n-1}(\BR )$
contains a subgroup $B\leq A$ such that
\begin{itemize}
\item
$|A/B|\leq i$,
\item
$B$ is contained in a connected abelian group.
\end{itemize}
\end{lem}

\begin{proof}
Dividing $\BR\times SO_{n-1}(\BR )$ by the factor $\BR$, we can prove the
analogous statement for $SO_{n-1}(\BR )$.
Let $\gO\subset SO_{n-1}(\BR )$ be a small symmetric identity neighborhood
with the following property: a subset $\gS\subset\gO^2$ generates a nilpotent
group if and only if
$\log\gS\subset \mathfrak{so}_{n-1}(\BR )=\mbox{Lie}\big( SO_{n-1}(\BR )\big)$ generates a
nilpotent Lie sub-algebra (c.f. \cite{Raghunathan} lemma 8.20).

Let $\mu$ be a Haar measure on $SO_{n-1}(\BR )$, and take
$$
 i=\frac{\mu\big( SO_{n-1}(\BR )\big)}{\mu ( \gO )}.
$$

Let now $A$ be a nilpotent subgroup of $SO_{n-1}(\BR )$.
Take $B=\langle A\cap\gO^2\rangle$. Then $B$ is contained in the connected
nilpotent group which corresponds to the nilpotent Lie subalgebra generated by
$\log A\cap\gO^2$. The closure of this connected nilpotent group is compact,
connected and nilpotent, hence it is abelian, and in particular so is $B$.

Now let $F=\{ a_j\}$ be a set of coset representatives for $A/B$. The sets
$a_j\cdot\gO$ must be disjoint, for otherwise we will have
$a_k^{-1}a_j\in\gO^2\cap A\subset B$ for some $i\neq j$.
Therefore
$$
 |F|\leq \frac{\mu\big( SO_{n-1}(\BR )\big)}{\mu ( \gO )}=i.
$$
\end{proof}

\begin{rem} \label{rmk}
The proof of Lemma \ref{index} shows that the analogous statement holds for the group $SO_{n}(\BR )$
as well. Below, we shall assume that the index $i$ is good for both groups
$\BR\times SO_{n-1}(\BR )$ and $SO_{n}(\BR )$.
\end{rem}

The maximal dimension of a connected abelian subgroup of $SO_{n-1}(\BR )$
is $[\frac{n-1}{2}]$ - the absolute rank of ${\Bbb S}{\Bbb O}_{n-1}$,
and hence the maximal dimension of a connected abelian subgroup of
$\BR\times SO_{n-1}(\BR )$ is $[\frac{n-1}{2}]+1$.
Moreover, the stabilizer group of a geodesic $c\subset X$ embeds naturally in
$\BR\times SO_{n-1}(\BR )$. This implies:

\begin{prop}\label{smalldim}
Assume $n\geq 4$.
Let $B\leq G$ be a connected abelian subgroup. If $B$ is not unipotent then
$\dim B <n-1$.
\end{prop}

\begin{proof}
By Lie's theorem, $B$ is triangulable over $\BC$. Thus, since $B$ is not
a unipotent group, there is a semisimple element $a\in B$.
The isometry $a$ is either elliptic or hyperbolic. If $a$ is hyperbolic, then as $B$ is abelian, any $b\in B$
preserves the axis of $a$. By the above paragraph we get
$$
 \dim B\leq [\frac{n-1}{2}]+1< n-1
$$
as $n\geq 4$.

If $a$ is elliptic then its set of fixed points $\text{Fix}(a)$ is a totally geodesic sub-symmetric space
of dimension $d=\dim \big(\text{Fix}(a)\big)$ which is $B$-invariant. Hence $B$ embeds in
$\SO_{n-d}\times\text{Isom}\big(\text{Fix}(a)\big)$ and the result follows by an inductive argument on the
dimension, keeping in mind that that any connected unipotent subgroup of $\text{Fix}(a)$ has dimension
$\leq d-1$.
\end{proof}

Let $p\in X(\infty )$ be the fixed point of $\gC_T$,
let $P\leq G$ be the parabolic group which corresponds to $p$
and let $P=MAN$ be the Langland's decomposition of $P$. Then $N$ is the
maximal unipotent group which corresponds to $p$, and $MN$ can be characterized
as the subgroup of $G$ which preserve the horospheres around $p$, and we have
$\gC_T=\gC\cap MN$, and $\gC_T^0 = \Gamma \cap N$. Additionally, $\gC_T$ (resp. $\gC_T^0$) is a cocompact
lattice in $MN$ (resp. in $N$) because $\gC_T\backslash N\cdot x$ is
homeomorphic to $T$ (where $x\in X$ is an arbitrary point and $N\cdot x$ the corresponding horosphere.
In particular any subgroup
of finite index $B\leq\gC_T^0$ is Zariski-dense in the unipotent group $N$,
and thus, such subgroup $B$ is not contained in a connected proper subgroup of
$N$. As $N$ acts simply transitively on each horosphere, its dimension is $n-1$.
So we see that $\text{span}\log (B )$ is $n-1$-dimensional and conclude:

\begin{clm}\label{bigdim}
Assume $n\geq 4$. No subgroup of finite index of $\gC_T$ is contained in a connected abelian
subgroup of $G$ of dimension $<n-1=\dim (N)$.
\end{clm}

The following claim is obvious.

\begin{clm}\label{12}
For a finitely generated group $\gD$, and a convergent sequence of
homomorphisms $f_n:\gD\to G$, the property: ``$f_n(\gD )$ is contained in a
connected abelian group of dimension $\leq m$'' is preserved by taking the limit.
\end{clm}

We can now prove the following main proposition:

\begin{prop}
Assume $n\geq 4$. If $\rho$ is a sufficiently small deformation of $\rho_0$, then the group $\rho (\gC_T)$ is contained in a parabolic group. Moreover, $\rho (\gC_T)$ contains no hyperbolic elements.
\end{prop}

\begin{proof}
Let $z_1\in\gC_T^0$ be a central element.
If $\rho (z_1)$ is parabolic then $\rho (\gC_T^0)$ is contained in the parabolic group corresponding to the unique fixed point of
$\rho (z_1)$ at $X(\infty )$.

Assume that $\rho (z_1)$ is not
parabolic. Then it is either elliptic or hyperbolic.
But if $\rho (z_1 )$ is hyperbolic then, as it is central,
$\rho (\gC_T^0 )$ is contained in the stabilizer group of the axis of
$\rho (z_1)$. However, by lemma
\ref{index} we obtain that $\rho (\gC_T^0)$ contains
a subgroup of index $\leq i$ which is contained in a connected abelian group,
which by Proposition \ref{smalldim} has dimension $<n-1$. Since $ \gC_T^0$ is
finitely generated (being a lattice in $N$) it has only finitely many subgroups
of index $\leq i$. Therefore the property ``to have a subgroup of index
$\leq i$ which is contained in a connected abelian group of dimension
$\leq n-2$'' is preserved by taking a limit. This, however, contradicts Claim \ref{bigdim}.

Suppose therefore that $\rho (z_1 )$ is elliptic. Then its set of fixed points is a sub-symmetric space $X_1$
of smaller dimension, on which $\rho (\gC_T^0)$ acts nilpotently, and we
continue by taking $z_2\in\gC_T^0$ which is central with respect to this action.
If $\rho (z_2 )$ is parabolic we conclude that $\rho (\gC_T^0)$ is contained in the parabolic group corresponding to the unique fixed point of
$\rho (z_2)$ at $X_1 (\infty ) \subset X(\infty )$.
If $\rho (z_2)$ is not parabolic then it cannot act parabolically on $X_1$. As above $\rho (z_2)$ can not acts hyperbolically.
If $z_2$ acts elliptically, we continue by induction, defining $X_2 = {\text Fix} (\rho (z_1 )) \cap {\text Fix} (\rho (z_2))$. By this way we conclude that $\rho (\gC_T^0)$ is contained in a parabolic group.

Now if all the $z_i$'s constructed above act elliptically, then $\rho (\gC_T^0)$ is contained in a compact group, and
(letting $\rho \rightarrow \rho_0$) we derive a
contradiction to Remark \ref{rmk}, Proposition \ref{smalldim} and Claim \ref{bigdim} as above. Therefore one of the $z_i$'s, say $z_{i_0}$,
must act as a parabolic. This implies that $\rho (\gC_T^0)$ has a unique fixed point at infinity, and since $\gC_T^0$ is normal in $\gC_T$
it follows that $\rho (\gC_T)$ fix this point and hence contained in the corresponding parabolic group.

Now if $\rho (\gC_T)$ would contain a hyperbolic element then so would $\rho (\gC_T^0)$ since a power of hyperbolic is hyperbolic.
However hyperbolic element can not commute with parabolic, while the action of any element on the corresponding subsymmetric space
$X_{i_0-1}$ commute with the action of the parabolic element $\rho (z_{i_0})$.
Therefore $\rho (\gC_T)$ does not contain hyperbolic elements.
\end{proof}

Since the action of $G$ at $X(\infty )$ is continuous, the unique fixed
point $x$ of $\rho (\gC_T)$ is close to the fixed point $x_0$ of $\gC_T$,
and since the orbit map $G\to G/P=X(\infty )$ is open,
there is some $g\in G$ near $1$ which takes $x_0$ to $x$. Conjugating $\rho$ by
$g$ we get $\rho^g$ which is again a small deformation. Moreover,
$\rho^g |_{\gC_T}$ is a deformation inside the parabolic subgroup $P$ which corresponds to the
cusp.

\begin{clm}
The group $\rho^g(\gC_T )$ is contained in $MN$.
\end{clm}

\begin{proof}
This follows immediately from the previous proposition since $MN$ is exactly the set of elements in
$P$ which are not hyperbolic.
\end{proof}

We shall now use Ehresmann-Thurston's principle (Proposition \ref{GWT}) to deduce:

\begin{prop}
The group $\rho^g(\gC_T)$ is a cocompact lattice in $MN$, and $\rho^g:\gC_T\to\rho^g(\gC_T)$ is an
isomorphism.
\end{prop}

Since $\rho^g(\gC_T)$ (and $\rho (\gC_T)$) is discrete and torsion free,
it acts properly and freely on each horosphere $N\cdot x$, and hence $\rho^g (\gC_T)\backslash N\cdot x$
(and hence also $\rho (\gC_T ) \backslash N^{g^{-1}} \cdot x$) is homeomorphic to
$\gC_T\backslash N\cdot x$ - the boundary of our cusp.
This  implies that $\rho (\gC_T)$ is the fundamental group of some canonical cusp
homeomorphic to $T\times [0,\infty )$.
This completes the proof of the following theorem and thus completes the rank one case:

\begin{thm}\label{rk-1thm}
Let $G$ be a connected center free simple Lie group of rank one neither isomorphic
to $PSL_2(\BR )$ nor to $PSL_2(\BC )$, and let $\gC\leq G$ be a non-uniform
lattice. Then there is a neighborhood $\gO$ of the inclusion $\rho_0:\gC\to G$
in the deformation space $\mathcal{R}(\gC ,G)$ such that for any $\rho\in\gO$,
$\rho(\gC )$ is a lattice in $G$, $\rho :\gC\to\rho (\gC )$ is an isomorphism,
and $X/\rho (\gC )$ is homeomorphic to $X/\gC$.
\end{thm}  

From Theorem \ref{rk-1thm} and Mostow rigidity, we conclude local rigidity in the rank one
case.

\medskip

It is well known that if $\gC$ is a locally rigid lattice in the group of real points
$G=\BG (\BR )$ of some $\BQ$-algebraic group $\BG (\BC )\leq GL_n(\BC )$, then it can be 
conjugated by an element $g\in G$ into $\BG (K)$ for some number field $K$.
This follows from the fact that $\gC$ is finitely generated, and hence, the deformation space
$\mathcal{R}(\gC ,G)$ has the structure of a real algebraic variety $V$ which is defined over 
$\BQ$. If we let  $\overline{\BQ}$ denote the algebraic closure of $\BQ$ in $\BC$, then this 
implies that $V(\overline{\BQ}\cap\BR )$ is dense in $V(\BR )$ in the real topology. 
Thus we can find a point $\rho\in V(\overline{\BQ}\cap\BR )$ arbitrarily close to the inclusion 
$\rho_0$. As $\gC$ is locally rigid, $\rho$ is given by conjugation, and as $\gC$ is finitely 
generated, $\rho (\gC )$ lies in some algebraic number field.

Let now $G\cong SL_2(\BC )$, and let $\gC\leq G$ be a (torsion free) non-uniform lattice. 
Then $\gC$ is not locally rigid. For this reason Garland and Raghunathan had to deal with this
case separately in \cite{Gar-Rag}, section 8. Using the above presented proof, one can
treat the $SL_2(\BC )$ case in the same manner as the general case, since, as explained above, 
if $\rho$ is a small deformation of $\gC$ which sends unipotents to unipotents, then 
$\rho (\gC )$ is a lattice and
$\rho :\gC\to\rho (\gC )$ is an isomorphism, and hence by Mostow rigidity, it is given by 
conjugation. In fact, the fundamental group of each cusp is finitely generated, and it is enough
to choose such finite generating set for each cusp, and to require that $\rho$ sends all these 
finitely many elements to unipotents. Then any unipotent in $\gC$, being conjugate to some element
of the chosen fundamental group of some cusp, must also be sent to a unipotent.
 
We can choose a large finite generating set $\gS$ for $\gC$ which contains a generating set 
for a chosen fundamental group for each of the finitely many cusps. Then, we add to the set of 
relations defining $\mathcal{R}(\gC ,G)$ the conditions that every unipotent in $\gS$ must be sent
to unipotent, and call this subspace $\mathcal{R}_U(\gC ,G)\subset\mathcal{R}(\gC ,G)$. 
Since, the unipotents are the zeros of the $\BQ$-polynomial $(\text{Ad}(g)-1)^n$ we see that the 
space $\mathcal{R}_U(\gC ,G)$ has the structure of a real algebraic variety defined over $\BQ$.
We conclude that:

\begin{cor}[Garland-Raghunathan]\label{SL_2(C)}
Let $\BG$ be a $\BQ$-algebraic linear group with $G=\BG (\BR )$ isomorphic to $PSL_2(\BC )$, and let 
$\gC$ be a non-uniform lattice in $G$. Then $\gC$ is conjugate to a subgroup of
$\BG (K)$ for some number field $K\leq \BR$.
\end{cor}

In the next section, we shall proceed by induction on $\mbox{rank} (G)$. The rank one case, thus,
would be the base step in this induction argument. We shall need the following weak
version of local rigidity, which holds in the general case (also for $G=SL_2(\BR )$), and is
proved by the same argument as above.

\begin{lem}
Let $G$ be a connected rank one simple Lie group with associated symmetric space $X$, and let 
$\gC\leq G$ be a torsion free lattice.
Let $\rho$ be a small deformation of $\gC$ which takes all the unipotents in $\gC$ to unipotents.
Then $\rho (\gC )$ is a lattice, $\rho (\gC )\backslash X$ is homeomorphic to $\gC\backslash X$,
and $\rho :\gC\to\rho (\gC )$ is an isomorphism.
\end{lem}


\section{The higher rank case}\label{higher-rank}

We shall now consider the case where $G$ is a connected center-free semisimple Lie
group of rank $\geq 2$ without compact factors, with associated symmetric space $X=G/K$,
and $\gC\leq G$ a non-uniform (torsion free) irreducible lattice, corresponding to the locally 
symmetric manifold $M=\gC\backslash X$.

As in the rank one case, we are going to apply Ehresmann-Thurston's principle. 
So we need to work with a compact
submanifold with boundary $M_0\subset M$ which exhausts most of $M$. We want to obtain $M_0$ from
$M$ by cutting out its ends, and to have enough information on the structure of the ends, so that
we can argue as we did in the rank one case. Of course, it would be nicer to give a simple proof
of local rigidity for general lattices, without assuming arithmeticity. 
However, we shall use the description of Leuzinger \cite{Leuzinger}
for the structure of $M\setminus M_0$, and hence assume, as in \cite{Leuzinger}, that
$\gC$ is arithmetic. 
We shall prove the following result, which originally is a consequence 
of Margulis super-rigidity theorem.

\begin{thm}
Let $\gC\leq G$ be a non-uniform arithmetic irreducible lattice, and let $\rho$ be a small 
deformation of the inclusion $\rho_0:\gC\to G$. Then $\rho (\Gamma )$ is a lattice, 
$\rho(\gC )\backslash X$ is homeomorphic to $M$, and $\rho :\gC\to\rho (\gC )$ is an isomorphism.
\end{thm}

The case of $\BQ$-$\mbox{rank}(\gC )=1$ is very similar to
the $\BR$-$\mbox{rank}$ one case. However, the picture is much more sophisticated
when $\BQ$-$\mbox{rank}(\gC )\geq 2$.
Then $M$ has only one end, i.e. if $M_0$ is defined appropriately then
$M\setminus M_0$ is connected. Moreover, ``from infinity'' $M$ ``looks like'' a simplicial complex
of dimension $\BQ$-$\mbox{rank}(\gC )$. Our argument will use an induction on the $\BR$-$\mbox{rank}$.

\medskip

Let $B$ be a parabolic subgroup of $G$ corresponding to a cusp of $\Gamma$, 
and $B=MAN$ be its Langland's decomposition over $\BQ$.
Let $\Gamma_0 := \Gamma \cap B$, $\Lambda = \Gamma \cap N$, and $C$ be the 
center of $\Lambda$. By arithmeticity, $\gL$ is a lattice in the unipotent group $N$, and 
in particular it is finitely generated.

\medskip

As in the rank one case, a crucial point in the proof is the following:

\begin{prop}\label{unipotent}
The group $\rho (\Lambda )$ is unipotent.
\end{prop}

In order to prove \ref{unipotent} we shall use the notion of $U$-elements 
from \cite{Gromov} and \cite{LMR}.

\begin{defn}
Let $\gD$ be a finitely generated group. An element $\gc\in\gD$ is called a 
{\it $U$-element} if the length of $\gc^n$, with respect to the word metric on 
$\gD$ corresponding to a fixed finite generating set, grows like $\log (n)$.
\end{defn}

\begin{rem}
The notion of $U$-elements is well defined and independent of the choice of 
the finite generating set.
\end{rem}

\begin{lem}\label{22}
If $\gD$ is a linear group and $\gc\in\gD$ is a $U$-element, then all the 
eigenvalues of $\gc$ are roots of unity.
\end{lem}

\begin{proof}
If $\gc\in\gD$ has an eigenvalue which is not a root of unity, then, as shown in the proof of Tits'
alternative \cite{Tits} (or more simply using Kronecker's lemma), 
there is a local field $k$ and a representation $r:\gD\to GL_m(k)$, such that 
$r(\gc )$ has an eigenvalue with absolute value $>1$. This implies that the norm of $r(\gc )^n$
in $GL_m(k)$ grows exponentially. If $\gS$ is a finite generating set, then the norms of the 
elements in $r(\gS)$ are uniformly bounded. As the norm is a sub-multiplicative function, this 
implies that the length of $\gc^n$ in the word metric grows linearly with $n$.
\end{proof}

\begin{thm}[Lubotzky-Mozes-Raghunathan \cite{LMR}]\label{L-M-R}
If $\gD\leq G$ is an irreducible lattice in a connected higher rank 
semisimple Lie group $G$, then any unipotent element
$\gc\in\gD$ is a $U$-element.
\end{thm}

Returning to our case, we conclude that:

\begin{cor}\label{U is unip}
If $\gc\in\gC$ is unipotent, and $\rho$ is sufficiently small, then $\rho (\gc )$ is unipotent.
\end{cor}

\begin{proof}
We know that $\gc$ is a $U$-element in $\gC$ by Theorem \ref{L-M-R}, and hence $\rho (\gc )$ is a 
$U$-element in $\rho (\gC )$. By lemma \ref{22} all the eigenvalues of $\rho (\gc )$ are roots
of unity. 

Recall that $\mathcal{R}(\gC ,G)$ is locally arcwise
connected, i.e. if $\rho$ is close enough to $\rho_0$ then they are connected by an
arc of deformations $\rho_t$. Since the eigenvalues depend continuously on the matrix, and since
the set of roots of unity is totally disconnected, we conclude that for any $t$, all the 
eigenvalues of $\rho_t(\gc )$ must be $1$. Hence $\rho_t(\gc )$, and in particular $\rho (\gc )$ 
is unipotent.
\end{proof}

\begin{cor}\label{4.7}
Given a finite generating set $\gS$ for $\gL$, there is a neighborhood 
$\gO$ of the inclusion $\rho_0:\gC\to G$, such that if $\rho\in\gO$ and 
$\gc\in\gS$, then $\rho (\gc )$ is unipotent.
\end{cor}

\begin{proof}[Proof of Proposition \ref{unipotent}]
Fix a generating set $\gS$ for $\gL$ and assume that $\rho\in\gO$ as in 
Corollary \ref{4.7}.
As $\rho (\gL )$ is nilpotent, its Zariski closure is nilpotent, and hence, 
by Lie's theorem, 
$A=\big(\overline{\rho (\gL )}^Z\big)^0$ is triangulable over $\BC$. 

Assume that $\overline{\rho (\gL )}^Z$ has $i$ connected components. Let
$\rho (\gL )^0=\langle \rho (\gc)^{i!} : \gc\in\gS\rangle$. 
Being a subgroup of $A$, $\rho (\gL )^0$ is triangulable over 
$\BC$. Since $\rho (\gS )$ is made of unipotents, $\rho (\gL )^0$ is
generated by unipotents, and hence it is a unipotent group. 
It follows that $\overline{\rho (\gL )^0}^Z$ is a unipotent algebraic group, and as such it is 
also connected.

Let $\gc$ be an element of $\gS$. Since $\rho (\gc )$ is unipotent 
$\overline{\langle\rho (\gc )\rangle}^Z=\overline{\langle\rho (\gc )^{i!} \rangle}^Z$.
This implies that $\rho (\gc )$ is actually contained in $\overline{\rho (\gL )^0}^Z$, 
and the proposition follows.
\end{proof} 

\begin{rem}
It might be possible to avoid the notion of $U$-elements and the use of Theorem \ref{L-M-R}, 
and to find a more elementary proof for Proposition \ref{unipotent}. 
This is not so hard in the special case where $G=SL_3(\BR )$ and $\gC\leq SL_3(\BR )$ 
is any lattice. Moreover, for the special case of
$G=SL_n(\BR )$ and $\gC =SL_n(\BZ )$ Theorem \ref{L-M-R} also has an elementary proof 
(see \cite{LMR2}).   
\end{rem}

\begin{clm}
The group $\rho (\gC_0 )$ fixes a point at infinity (i.e. in $X(\infty )$).
\end{clm}

\begin{proof}
According to \cite{BGS}, appendix 3, since $\rho (\gL )$ is unipotent, the set 
of common fixed points at $X(\infty )$
$$
{\mathcal B}= 
\cap_{\gc\in\gL}
\mbox{Fix} \big(\rho (\gamma )\big) 
$$
is non-empty, and so is the set 
$$
{\mathcal B}_0=
\{ z\in {\mathcal B}:Td(z,y)\leq \frac{\pi}{2},~\mbox{for any}~y\in{\mathcal B}\},
$$
(here $Td$ denotes the Tits distance on $X(\infty )$)
which clearly has diameter $\leq\pi /2$ and thus has a unique Chebyshev center $O$.
Since $\gC_0$ normalizes $\gL$, the point $O$ 
must be invariant under the action of $\rho (\gC_0 )$.
Therefore $\rho (\gC_0 )$ is contained in the parabolic subgroup which corresponds to $O$.
\end{proof}

\medskip

Let $P$ be a parabolic subgroup which contains $\rho (\gC_0)$, and let
$P= M_P A_P N_P$ be a Langland's decomposition of $P$. 
Note that the group $N_P$ is well defined, the group $M_P$ is 
determined up to conjugation by elements of $N_P$, and the group $M_P N_P$ is 
well defined. The group $M_P N_P$ can be characterized as the set of all elements in $P$
which acts unimodularly on the Lie algebra $\mbox{Lie}(N_P)$ through the adjoint action, or as
the set of all elements in $P$ which preserve the Busemann functions and their level sets - 
the horospheres around the corresponding fixed points at $X(\infty )$.
In fact we can choose $P$ so that
\begin{enumerate}
\item $\rho ( \Lambda ) \subset N_P$, and
\item $N_P$ is the Zariski closure $N_P=\overline{\rho ( \Lambda )}^Z$.
\end{enumerate}

\medskip

We need the following analogue of Claim \ref{12}:

\begin{clm}
For a finitely generated group $\gD$, and a convergent sequence of 
homomorphisms $f_n:\gD\to G$, the property ``$f_n(\gD )$ is contained in a 
connected unipotent group of dimension $\leq m$'' is preserved by taking a limit.
\end{clm}

This claim can be easily proved using the fact that for connected unipotent groups the exponential 
map is a diffeomorphism.

\medskip

It follows that if $\rho$ is sufficiently small:

\begin{enumerate}
\item
$\dim (N_P)\geq\dim (N)$,
\item
$P$ must be conjugate to a parabolic which contains $B$. 
\end{enumerate}

The second statement follows from the fact that $\gC_0$ is a an arithmetic subgroup and hence
a lattice in $MN$ and in particular its Zariski closure contain $M'N$, where $M'$ is the product of all
non-compact factors of $M$. 
Hence, when $\rho$ is small enough, $P$ must contain a conjugate of $M'N$, but this 
implies that it contains a conjugate of $B$.

Since the inclusion relation is opposite for parabolic subgroups and for their
unipotent radicals, these two facts together imply:

\begin{cor}
If $\rho$ is sufficiently small, $P$ is conjugate to $B$.
\end{cor} 

Since the canonical map $G\to G/B$ is continuous and open, we have $P=g^{-1}Bg$ for some $g\in G$ 
close to $1$, and by replacing $\rho$ by its conjugation with $g$, $\rho^g$ (which is still a small 
deformation of $\rho_0$), we may assume that $P=B$.

\medskip

We conclude that

\begin{clm}
The group $\rho^g( \Lambda )$ is a (cocompact) lattice in $N$.
\end{clm}

This is a subsequent of the fact that $\rho^g$ induces a small deformation of 
$\Lambda$ (a cocompact lattice) in $N$ together with either some classical theorems on lattices of nilpotent 
Lie groups (Ma\'lcev theory \cite{malcev}) or, 
more simply, Ehresmann-Thurston's principle. 

The next step is:
 
\begin{clm}
The group $\rho (\Gamma_0 )$ is contained in $M_P N_P$.
\end{clm}

This is because the group $\rho (\Gamma_0 )$ normalizes the cocompact lattice $\rho (\gL )$ of 
$N_P$, and hence $\mbox{Ad}\big(\rho (\Gamma_0 )\big)$ acts unimodularly on $\mbox{Lie}(N_P )$.

\medskip

As noted above $\Gamma_0$ is a lattice in $M N$. Moreover, the projection of 
$\Gamma_0$ in $M$ is a lattice in $M$, 
let's denote this lattice by $L$. The morphism $\rho^g$, composed with the projection $M N\to M$, 
induces a small deformation $\tilde{\rho^g}(L)$
of $L$ in $M$. Moreover, by Theorem \ref{L-M-R} and Corollary \ref{U is unip}, $\rho^g$ and hence
$\tilde{\rho^g}$ takes unipotents to unipotents. Hence, by an induction procedure on the real rank,
starting with rank one, of which we already took care in the previous section, we conclude 
that:

\begin{clm}
The group $\tilde{\rho^g}(L)$ is a lattice in $M$.
\end{clm} 

Since the Haar measure of $MN$ is the product of the Haar measures of $M$ and of $N$, 
and since $\gL$ is a cocompact lattice in $N$, we get:

\begin{cor}
The group $\rho (\gC_0 )$ is a lattice in $MN$.
\end{cor}

\medskip

To conclude the proof, we now need, as in the rank one case, to introduce a notion of cusp, but this 
time it will depend on
the group $\Gamma$. By \cite{Leuzinger}, theorem 4.2, on the locally symmetric manifold
$M= \Gamma \backslash X$, there exists a continuous and piecewise real analytic exhaustion 
function $h:M \rightarrow [0, \infty )$ such that, for any $s\geq 0$, the sub-level set 
$V(s) : = \{ h \leq s \}$ is a compact submanifold with corners of $V$. Moreover, 
the boundary of $V(s)$, which is a level set of $h$, consists of projections of subsets of 
horospheres in $X$. 

We can choose a sufficiently large $s_0$ in order to ensure that for any $s > s_0$, 
$V(s)-V(s_0 )$ is a collar neighborhood of the boundary of $V(s_0 )$ inside $V(s)$.
In the following we let $M_0 = V(s)$ for some real $s > s_0$.

By the Ehresmann-Thurston principle, there is a $(G,X)$-structure $M_0 '$ on $M_0$ whose holonomy is 
$\rho$ and which induces a new $(G,X)$-structure on $V(s)-V(s_0 )$. 
The fundamental group of $V(s)-V(s_0 )$ is a finite complex of groups $\mathcal{C}$ where each 
simplex group is the intersection $\Gamma \cap M_P N_P$ for some $\BQ$-parabolic subgroup 
$P$ of $G$ and the inclusion order for simplices is the reversed inclusion order for the 
corresponding parabolic groups. In particular the vertex groups correspond to the $\gC$-conjugacy
classes of maximal $\BQ$-parabolics. 

The proof of proposition 3.3 in \cite{Leuzinger} implies that any finite complex of groups as 
above leaves invariant a countable union of horospheres of $X$ (level sets of Busemann functions) 
centered at the points of the ideal boundary of 
$X$ which corresponds to the vertices of the complex. In fact, Leuzinger proves that 
the developing map of $M_0$ embeds $\tilde{M}_0$ in $X$ as the complement $X(s)$ of 
countable union of open horoballs.
These horoballs are disjoint if and only if $\Gamma$ is an arithmetic subgroup of a $\BQ$-$\mbox{rank}$ 
one group. The projection 
$\pi:X \rightarrow M$ maps $X(s)$ to $V(s)$ which is a compact manifold with corners whose 
fundamental group is isomorphic to 
$\Gamma$. The image of the developing map of $V(s) -V(s_0 )$ with the $(G,X)$-structure induced 
from $M_0$ is $X(s) -X(s_0 )$. It is a $\pi_1 (V(s) -V(s_0 ))$-invariant subset of $X$.

The image $Y$ of the developing map of $V(s)-V(s_0 )$ with the $(G,X)$-structure
induced from $M_0 '$ is a subset of $X$ invariant under $\rho \big(\pi_1 (V(s)-V(s_0 )\big)$ which converges toward
$X(s)-X(s_0 )$ on every compact subset as $\rho$ tends to $\rho_0$.
According to the above claims, when $\rho$ is sufficiently small,
the group $\rho \big(\pi_1 \big( V(s)-V(s_0 )\big)\big)$ is still a complex of 
groups isomorphic to $\mathcal{C}$, and each simplex of the new complex is associated with a 
parabolic group $B'$ of the same type as the parabolic $B$ associated to the
corresponding simplex of $\mathcal{C}$. Moreover $\rho :(\gC\cap B)\to(\rho (\gC )\cap B')$ is an 
isomorphism, its image $\rho (\gC )\cap B'$ is contained in $M_{B'} N_{B'}$ and it is a lattice 
there. The group $\rho \big(\pi_1\big(V(s) -V(s_0 )\big)\big)$ thus leaves invariant a countable 
family of horospheres as above. More 
precisely, we get that for $\rho$ sufficiently close to $\rho_0$, the image $Y$ of the developing 
map of $V(s) - V(s_0)$ with the $(G,X)$-structure induced from $M_0 '$ contains the subset of $X$ 
which is bounded between two families of horospheres as above. By adding
the corresponding countably many horoballs, $Y$ can be completed by a 
$\rho\big(\pi_1\big(V(s) -V(s_0 )\big)\big)$-invariant subset of $X$, whose boundary is one part 
of the boundary of $Y$ with the reversed orientation and 
whose quotient $C$ by $\rho\big(\pi_1\big(V(s)-V(s_0)\big)\big)$ has a finite volume. 
As the restriction of 
the $(G,X)$-structures from $M_0 '$ and from $C$ coincide on $V(s)-V(s_0 )$, these two sets can 
be glued along $V(s) -V(s_0 )$ in order to yield a complete $(G,X)$-manifold $M'$ whose 
fundamental 
group is $\rho (\Gamma )$. Note that $M'$ has finite volume, being a union of a compact set and a 
finite volume set. 
The finiteness of the volume of each of the (finitely many) parts composing the ``cusp'' 
$M'\setminus M_0'$ can be proved in the same way as one proves the finiteness of the volume of 
an ordinary Siegel set. Namely, by expressing the Riemannian measure of $X$ in terms of the 
Haar measures of $M_P, N_P$ and $A_P$, and using Fubiny theorem, integrating the function $1$
first along the horosphere corresponding to $P$ in $V(s_0)$ (here we get a finite number since
$\rho (\gC )\cap M_PN_P$ is a lattice in $M_PN_P$), and then in the directions of the end - 
the simplex of the Tits' building which corresponds to $P$. Note that the Riemannian measure of $X$
is obtained from the product measure of the corresponding horosphere and the Haar measure of $A_P$
by multiplying by a factor which tends to zero exponentially fast as one moves toward the end.

\medskip

\begin{footnotesize}
Nicolas Bergeron: Unit\'e Mixte de Recherche 8628 du CNRS, Laboratoire de Math\'ematiques, B\^at. 425, Universit\'e 
Paris-Sud, 91405 Orsay Cedex, France \\
E-mail address: \texttt{Nicolas.Bergeron@math.u-psud.fr} \\

Tsachik Gelander: Dept. of Math, Yale University, 10 Hillhouse Ave, New Haven CT 06520, USA \\
E-mail address: \texttt{tsachik.gelander@yale.edu}

\end{footnotesize}

\end{document}